\documentclass[11pt]{amsart}
\usepackage{amsmath, amssymb}
\usepackage{amsfonts}
\usepackage{graphicx}
\usepackage{mathrsfs}
\usepackage[arrow,matrix,curve,cmtip,ps]{xy}
\usepackage{amsthm}
\usepackage{enumerate}
\usepackage{color}
\usepackage[colorlinks]{hyperref}
\usepackage{tikz}

\allowdisplaybreaks

\newtheorem{thm}{Theorem}[section]
\newtheorem{lem}[thm]{Lemma}
\newtheorem{prop}[thm]{Proposition}
\newtheorem{cor}[thm]{Corollary}
\newtheorem*{theorem*}{Theorem}
\theoremstyle{remark}
\newtheorem{rem}[thm]{Remark}

\newtheorem{defn}[thm]{Definition}
\newtheorem{ex}[thm]{Example}

\numberwithin{equation}{section}

\begin{document}
\title[Prime and primitive ideals  of ultragraph  Leavitt path algebras]{Prime and primitive ideals  of ultragraph  Leavitt path algebras}

\author[A. Pourabbas, M. Imanfar  and H. Larki]{A. Pourabbas, M. Imanfar and H. Larki}

\address{Faculty of Mathematics and Computer Science,
	Amirkabir University of Technology, 424 Hafez Avenue, 15914
	Tehran, Iran.}
\email{arpabbas@aut.ac.ir, m.imanfar@aut.ac.ir}
\address{Department of Mathematics, Faculty of Mathematical Sciences and Computer,
	Shahid Chamran University, Ahvaz, Iran}
\email{h.larki@scu.ac.ir}

%\thanks{This research was supported by NSF Grant DMS-1234567}

%\date{\today}



\begin{abstract}
	Let $\mathcal G$ be an ultragraph and let $K$ be a field. We describe prime  and primitive ideals in the ultragraph Leavitt path algebra $L_K(\mathcal G)$. We identify the  graded prime ideals in terms of  downward directed sets and then we characterize the non-graded prime ideals. We show that the non-graded prime ideals of $L_K(\mathcal G)$ are always primitive. 
\end{abstract}

\maketitle

%%%%%%%%%%%%%%%%%%%%%%%%%%%%%%%%%%%%%%%%%%%%%%%%%%%%%%%%%%%%%%%%%%%%%%
%%%%%%%%%%%%%%%%%%%%%%%%%%%%%%%%%%%%%%%%%%%%%%%%%%%%%%%%%%%%%%%%%%%%%%
\section{Introduction}

Let $E$ be a (directed) graph. The Leavitt path algebra $L_K(E)$, which is a purely algebraic analogue of graph Cuntz-Krieger $C^*$-
algebra $C^*(E)$ \cite{kum1,fow}, was introduced in \cite{abr,abr1}. The algebras $L_K(E)$ are generalizations of the Leavitt algebras $L(1,n)$ \cite{lea}. The study of the structure of prime and primitive Leavitt path algebras have been the subject of a series of papers in recent years (see \cite{arp1,arp2,abr2}). For a  unital commutative ring $R$, the graded prime (primitive) ideals of $L_R(E)$ are characterized in \cite{lar1} via special subsets of the vertex (called maximal tails). Furthermore, the structure of  non-graded prime and primitive ideals of $L_K(E)$ have been identified in \cite{kul}. It was shown in \cite{kul} that there is a one-to-one  correspondence between non-graded prime (primitive) ideals of $L_K(E)$ and maximal tails containing a loop without exits and the prime spectrum of $K[x,x^{-1}]$.

Ultragraph Leavitt path algebras have been widely studied, see \cite{gon,fir,cas,gon1,gon2,nam,haz}. Ultragraph Leavitt path algebra $L_R(\mathcal G)$ was introduced in \cite{ima} as the algebraic version of ultragraph $C^*$-algebra $C^*(\mathcal G)$ \cite{tom1}. The algebras $L_R(\mathcal G)$ are generalizations of the Leavitt path algebras $L_R(E)$. The structure of ultragraph Leavitt path algebras are more complicated, because in ultragraphs the range of each edge is allowed to be a nonempty set of vertices rather than a single vertex. The class of ultragraph Leavitt path algebras is strictly larger than the class of Leavitt path algebras of directed graphs.  Also, every Leavitt path algebra of a directed graph can be embedded as a subalgebra in a unital ultragraph Leavitt path algebra.

The aim of this paper is to  give a complete description of the prime ideals as well as the primitive ideals of $L_K(\mathcal G)$. We start in Section \ref{s2} by recalling the definition of the quotient ultragraph $\mathcal G/(H,S)$ and its Leavitt path algebra $L_K\big(\mathcal{G}/(H,S)\big)$. In Section \ref{s3}, we characterize the graded prime ideals in terms of the downward directed sets. To describe the structure of non-graded prime ideals,  we investigate the structure of the closed ideals of $L_K\big(\mathcal{G}/(H,S)\big)$ which contain no nonzero set idempotents. In Section \ref{s4}, we give a complete description of primitive ideals. We show that a graded prime ideal $I_{(H,B_H)}$ is primitive if and only if the quotient ultragraph $\mathcal{G}/(H,B_H)$ satisfies Condition (L). Finally, we prove that every non-graded prime ideal in $L_K(\mathcal G)$ is primitive.

%%%%%%%%%%%%%%%%%%%%%%%%%%%%%%%%%%%%%%%%%%%%%%%%%%%%%%%%%%%%%%%%%%%%%%
%%%%%%%%%%%%%%%%%%%%%%%%%%%%%%%%%%%%%%%%%%%%%%%%%%%%%%%%%%%%%%%%%%%%%%
\section{Preliminaries}\label{s2}

In this section, we briefly review the basic definitions and properties of
ultragraphs \cite{tom1}, quotient ultragraphs \cite{lar2} and ultragraph Leavitt path algebras \cite{ima}.

An \emph{ultragraph} $\mathcal{G}=(G^0,\mathcal{G}^1,r_{\mathcal{G}},s_{\mathcal{G}})$ consists of a countable set of vertices $G^0$, a countable set of edges $\mathcal{G}^1$, the source map $s_{\mathcal{G}}:\mathcal{G}^1\rightarrow G^0$ and the range map $r_{\mathcal{G}}:\mathcal{G}^1 \rightarrow \mathcal{P}(G^0)\setminus \{\emptyset\}$, where $\mathcal{P}(G^0)$ denotes the collection of all subsets of $G^0$. By an \emph{algebra} in $\mathcal P(X)$, we mean a collection of subsets of $X$ which is closed  under the set operations $\cup$, $\cap$ and $\setminus$. We write $\mathcal{G}^0$ for the smallest algebra in $\mathcal P(G^0)$ containing $\big\{\{v\},r_{\mathcal{G}}(e):v\in G^0\,\,\mathrm{and}\,\, e\in\mathcal{G}^1\big\}$.

\begin{defn}
	A subcollection $H\subseteq\mathcal{G}^0$ is  \emph{hereditary} if
	\begin{enumerate}
		\item $\{s_{\mathcal{G}}(e)\}\in H$ implies $r_{\mathcal{G}}(e)\in H$ for all $e\in\mathcal{G}^1$.
		\item $A\cup B\in H$ for all $A,B\in H$.
		\item $A\in H$, $B\in\mathcal{G}^0$ and $B\subseteq A$, imply $B\in H$.
	\end{enumerate}
\end{defn}
The subcollection $H\subseteq\mathcal{G}^0$ is \emph{saturated} if, whenever $0<|s_{\mathcal{G}}^{-1}(v)|<\infty$ satisfies $r_{\mathcal{G}}(e)\in H$ for all $e\in s^{-1}_{\mathcal{G}}(v)$, we have $\{v\}\in H$. Let $H\subseteq\mathcal{G}^0$ be a  saturated hereditary subcollection. We define the \emph{breaking vertices} of $H$ to be the set 
$$B_H:=\Big\{ v\in G^0: \big|s_{\mathcal{G}}^{-1}(v)\big|=\infty \,\,\mathrm{and}~~ 0<\big|s_{\mathcal{G}}^{-1}(v)\cap \{e:r_{\mathcal{G}}(e)\notin H\}\big|<\infty\Big\}.$$
If $H\subseteq\mathcal{G}^0$ is a saturated hereditary subcollection and $S\subseteq B_H$, then $(H,S)$ is called an \emph{admissible pair} in $\mathcal{G}$.

\subsection{Quotient ultragraph} In order to define the quotient of ultragraphs we need to recall and introduce some notations. Let $(H,S)$ be an admissible pair in ultragraph $\mathcal{G}=(G^0,\mathcal{G}^1,r_{\mathcal G},s_{\mathcal G})$. Given $A\in\mathcal P(G^0)$, set $\overline{A}:=A\cup \{w':w\in A\cap (B_H\setminus S)\}$. The smallest algebra in $\mathcal P(\overline{G^0})$ containing 
$$\big\{\{v\}, \{w'\}, \overline{r_{\mathcal{G}}(e)}: v\in G^0, w\in B_H\setminus S~\mathrm{and}~e\in\mathcal{G}^1\big\}$$
is denoted by $\overline{\mathcal{G}^0}$. Let $\sim$ be a relation on $\overline{\mathcal{G}^0}$ defined by $A\sim B$ if and only if there exists $V\in H$ such that $A\cup V=B\cup V$.  Then, by \cite[Lemma 3.5]{lar2}, $\sim$ is an equivalent relation on $\overline{\mathcal{G}^0}$ and the operations
$$[A]\cup[B]:=[A\cup B], ~ [A]\cap [B]:=[A\cap B] ~ \mathrm{and}~ [A]\setminus [B]:= [A\setminus B]$$
are well-defined on the equivalent classes $\{[A]: A\in \overline{\mathcal{G}^0}\}$. It can be shown that $[A]=[B]$ if and only if both $A\setminus B$ and $B\setminus A$ belong to $H$.
\begin{defn}
	Let $(H,S)$ be an admissible pair in $\mathcal{G}$. The \emph{quotient ultragraph of $\mathcal{G}$ by $(H,S)$} is the quadruple $\mathcal{G}/(H,S):=(\Phi( G^0),\Phi(\mathcal{G}^1),r,s)$, where
	$$\Phi( G^0):=\big\{[\{v\}],[\{w'\}]:v\in G^0,\{v\}\notin H~\mathrm{and}~w\in B_H\setminus S\big\},$$
	$$\Phi(\mathcal{G}^1):=\big\{e\in{\mathcal{G}}^1:{ r_{\mathcal{G}}}(e)\notin H\big\},$$
	and $s: \Phi(\mathcal{G}^1) \rightarrow \Phi( G^0)$ and $ r:\Phi(\mathcal{G}^1) \rightarrow \{[A]:A\in\overline{\mathcal{G}^0}\}$ are the maps defined by $s(e):=[\{s_{\mathcal{G}}(e)\}]$ and $r(e):=[\overline{r_{\mathcal{G}}(e)}]$
	for every $e\in\Phi(\mathcal{G}^1)$, respectively.
\end{defn}

We  denote by $\Phi(\mathcal G^0)$  the smallest algebra in $\{[A]:A\in\overline{\mathcal{G}^0}\}$ containing 
$$\Phi( G^0)\cup\Big\{r(e): e\in  \Phi(\mathcal{G}^1) \Big\}.$$
One can see that $\Phi(\mathcal G^0)=\big\{[A]:A\in\overline{\mathcal{G}^0} \big\}$. If $A,B\in\overline{\mathcal{G}^0}$ and $[A]\cap [B]=[A]$, then we write $[A]\subseteq [B]$. Also,  we write $[v]$ instead of $[\{v\}]$ for every vertex $v\in G^0\setminus H$.

A \emph{path} in $\mathcal{G}/(H,S)$ is a finite sequence $\alpha=e_1e_2\cdots e_n$ of edges with $s(e_{i+1})\subseteq r(e_i)$ for $1\leq i\leq n-1$. We consider the elements of $\Phi(\mathcal G^0)$ as the paths of length zero. We let $\mathrm{Path}\big(\mathcal{G}/(H,S)\big)$ denotes  the set of all paths in $\mathcal{G}/(H,S)$. We define $[A]^*:=[A]$ and $\alpha^*:=e_n^*e_{n-1}^*\cdots e_1^*$, for every $[A]\in\Phi(\mathcal G^0)$ and $\alpha=e_1e_2\cdots e_n\in\mathrm{Path}\big(\mathcal{G}/(H,S)\big)$. The maps $r,s$ extend to $\mathrm{Path}\big(\mathcal{G}/(H,S)\big)$ in an obvious way.

\subsection{Leavitt path algebra} A vertex $[v]\in\Phi( G^0)$ is called an \emph{infinite emitter} if $|s^{-1}([v])|=\infty$  and a \emph{sink} if $|s^{-1}([v])|=\emptyset$. A \emph{singular vertex} is a vertex that is either a sink or an infinite emitter and we denote the set of singular vertices  by $\Phi_{\mathrm{sg}}( G^0)$.

\begin{defn}\label{deflp}
	Let $\mathcal{G}/(H,S)$ be a quotient ultragraph and let  $K$ be a field. A \emph{Leavitt $\mathcal{G}/(H,S)$-family} in a $K$-algebra $X$ is a set $\{q_{[A]},t_e,t_{e^*}:[A]\in\Phi(\mathcal G^0)\,\,\mathrm{and}\,\,e\in\Phi(\mathcal G^1)\}$ of elements in $X$ such that
	\begin{itemize}
		\item[(1)]$q_{[\emptyset]}=0$, $q_{[A]} q_{[B]}=q_{[A]\cap [B]}$ and $q_{[A]\cup [B]}=q_{[A]} +q_{[B]}-q_{[A]\cap [B]}$;
		\item[(2)] $q_{s(e)}t_e=t_eq_{r(e)}=t_e$ and $q_{r(e)}t_{e^*}=t_{e^*}q_{s(e)}=t_{e^*}$;
		\item[(3)] $t_{e^*}t_f=\delta_{e,f}q_{r(e)}$;
		\item[(4)] $q_{[v]}=\sum_{s(e)=[v]}t_e t_{e^*}$ whenever $[v]\in\Phi( G^0)\setminus \Phi_{\mathrm{sg}}( G^0)$.
	\end{itemize}
	The \emph{Leavitt path algebra of $\mathcal{G}/(H,S)$}, denoted by $L_K\big(\mathcal{G}/(H,S)\big)$, is defined to be the $K$-algebra generated by a universal Leavitt $\mathcal{G}/(H,S)$-family.
\end{defn}

 Let $\mathcal{G}$ be an ultragraph. If we consider the quotient ultragraph $\mathcal{G}/(\emptyset,\emptyset)$, then $[A]=\{A\}$  For every $A\in\mathcal{G}^0$. Thus we can consider the ultragraph $\mathcal{G}$ as the quotient ultragraph $\mathcal{G}/(\emptyset,\emptyset)$. So it makes sense to talk about the \emph{ultragraph Leavitt path algebra $L_K(\mathcal G)$} and define it as $L_K\big(\mathcal{G}/(\emptyset,\emptyset)\big)$. In fact, the definition of ultragraph Leavitt path algebras (\cite[Defintion 2.1]{ima}) is an special case of the Definition \ref{deflp}.
 
 By \cite[Theorem 2.15]{ima}, $L_K\big(\mathcal{G}/(H,S)\big)$ is of the
 form
 $$\mathrm{span}_K\big\{t_\alpha q_{[A]}t_{\beta^*}: \alpha,\beta\in \mathrm{Path}\big(\mathcal{G}/(H,S)\big) ~\mathrm{and}~ r(\alpha)\cap [A]\cap r(\beta)\ne[\emptyset]  \big\},$$
 where $t_\alpha:=t_{e_1}t_{e_2}\cdots t_{e_n}$ if $\alpha=e_1e_2\cdots e_n$ and $t_\alpha:=q_{[A]}$ if $\alpha=[A]$. Also, $L_K\big(\mathcal{G}/(H,S)\big)$ is a $\mathbb{Z}$-graded ring by the grading
 $$L_K\big(\mathcal{G}/(H,S)\big)_n=\mathrm{span}_K\big\{t_{\alpha}q_{[A]}t_{\beta^*}:|\alpha|-|\beta|=n\big\} \hspace{.5cm} (n\in\mathbb{Z}).$$
 
 Throughout the article we denote the universal Leavitt $\mathcal{G}$-family and $\mathcal{G}/(H,S)$-family by $\{s,p\}$ and $\{t,q\}$, respectively. Also, we suppose that $L_K(\mathcal G)=L_K(s,p)$ and $L_K\big(\mathcal{G}/(H,S)\big)=L_K(t,q)$.

 %%%%%%%%%%%%%%%%%%%%%%%%%%%%%%%%%%%%%%%%%%%%%%%%%%%%%%%%%%%%%%%%%%%%%%
 %%%%%%%%%%%%%%%%%%%%%%%%%%%%%%%%%%%%%%%%%%%%%%%%%%%%%%%%%%%%%%%%%%%%%%
 \section{Prime ideals}\label{s3}

  In this section, we give a complete description of the prime ideals of $L_K(\mathcal G)$. We first characterize the primeness of a graded ideal in terms of the downward directed sets and then we characterize the non-graded prime ideals of $L_K(\mathcal G)$.

 \subsection{Graded prime ideals}
  We recall the definition of downward directed sets from \cite[Definition 5.3]{lar2}. Let $\mathcal{G}$ be an ultragraph and $A,B\in\mathcal G^0$. We write $A\geq B$ if either $B\subseteq A$ or there is a path $\alpha$ of positive length such that $s_{\mathcal{G}}(\alpha)\in A$ and $B\subseteq r_{\mathcal{G}}(\alpha)$. For the sake of simplicity, we will write $v\geq w$ instead of $\{v\}\geq\{w\}$. A nonempty subset $M$ of $\mathcal{G}^0$ is said to be \emph{downward directed} if for every $A,B\in M$ there exists $\emptyset\ne C\in M$ such that $A,B\geq C$.
  
  \begin{lem}\label{lemgp}
  	Let $\mathcal{G}/(H,S)$ be a quotient ultragraph. Then every nonzero graded ideal of $L_K\big(\mathcal{G}/(H,S)\big)$ contains idempotent $q_{[A]}$ for some $[\emptyset]\ne[A]\in \Phi(\mathcal G^0)$. 
  \end{lem}
  \begin{proof}
  	Let $I$ be a nonzero graded ideal of $L_K\big(\mathcal{G}/(H,S)\big)$. Then the quotient map $\pi:L_K\big(\mathcal{G}/(H,S)\big)\rightarrow L_K\big(\mathcal{G}/(H,S)\big)/I$ is a graded homomorphism. Suppose that $q_{[A]}\notin I$ for every $[\emptyset]\ne[A]\in \Phi(\mathcal G^0)$. Then, by \cite[Theorem 3.2]{ima}, $\phi$ is injective, which is impossible.
  \end{proof}
  
  \begin{lem}\label{lempd}
  	Let $I$ be an ideal of $L_K(\mathcal G)$. Consider $H_I:=\{A\in\mathcal{G}^0:p_A\in I\}$. If $I$ is prime, then $\mathcal{G}^0\setminus H_I$ is downward directed.
  \end{lem}
  \begin{proof}
  	 Let $X:=L_K(\mathcal G)/I$ and denote by $\widetilde x$ the image of $x\in C^*(\mathcal G)$ in $X$. For every $A,B\in \mathcal{G}^0\setminus H_I$, both $X\widetilde{p}_AX$ and $X\widetilde{p}_BX$ are nonzero ideals of $X$. Suppose that $I$ is a prime ideal of $L_K(\mathcal G)$. It follows that $X$ is a prime ring. Thus $X\widetilde{p}_AX\widetilde{p}_BX$ is a nonzero ideal of $X$ and consequently $\widetilde{p}_AX\widetilde{p}_B\ne\{0\}$. Since $X$ is of the form
  	 $$\mathrm{span}_K\big\{\widetilde s_\alpha\widetilde p_{C}\widetilde s_{\beta^*}:C\in\mathcal G^0, \alpha,\beta\in \mathrm{Path}(\mathcal G) ~\mathrm{and}~ r_{\mathcal G}(\alpha)\cap C\cap r_{\mathcal G}(\beta)\ne\emptyset\big\},$$
  	 there exist $\alpha,\beta\in \mathrm{Path}(\mathcal G)$ and $C\in\mathcal{G}^0$ such that ${p}_A(s_\alpha  p_{C}  s_{\beta^*}){p}_B\ne0$, which would mean that $s_{\mathcal{G}}(\alpha)\in A$ and $s_{\mathcal{G}}(\beta)\in B$. Thus if we set $D:=r_{\mathcal{G}}(\alpha)\cap C\cap r_{\mathcal{G}}(\beta)$, then we can deduce that $A,B\geq D$. Therefore $\mathcal{G}^0\setminus H_I$ is downward directed.
  \end{proof}

  	Let $(H,S)$ be an admissible pair in  $\mathcal{G}$. For any $w\in B_H$, set
  	$$p_w^H:=p_w-\sum_{s_{\mathcal G}(e)=w,~ r_{\mathcal{G}}(e)\notin H}s_e s_e^*,$$
  	and we define $I_{(H,S)}$ as the (two-sided) ideal of $L_K(\mathcal G)$ generated by the idempotents $\{p_A:A\in H\} \cup \left\{p_w^H:w\in S\right\}$. By \cite[Theorem 4.4]{ima}, $L_K\big(\mathcal{G}/(H,S)\big) \cong L_K(\mathcal G)/I_{(H,S)}$ and the correspondence $(H,S)\mapsto I_{(H,S)}$ is a bijection from the set of all admissible pairs of $\mathcal{G}$  to the set of all graded ideals of $L_K(\mathcal G)$.

  \begin{prop}\label{probn}
  	If $I_{(H,S)}$ is a prime ideal of $L_K(\mathcal G)$, then $|B_H\setminus S|\leq1$
  \end{prop}
\begin{proof}
	Assume to the contrary that $w,z\in B_H\setminus S$. Let $I$ and $J$ be two ideals of $L_K\big(\mathcal{G}/(H,S)\big)$ generated by $q_{[w']}$ and $q_{[z']}$, respectively. Since $[w'],[z']$ are two sinks in $\mathcal{G}/(H,S)$ and $q_{[w']}q_{[z']}=0$, we have that $q_{[w']}L_K\big(\mathcal{G}/(H,S)\big) q_{[z']}=0$. Thus $IJ={0}$, contradicting the primeness of $L_K\big(\mathcal{G}/(H,S)\big)$.
\end{proof}

\begin{thm}\label{thmgp}
	Let $\mathcal{G}$ be an ultragraph. Set
	$$X_1=\{I_{(H,B_H)}:\mathcal{G}^0\setminus H\mathrm{~is~downward~ directed}\}$$
	and
	$$X_2=\{I_{(H,B_H\setminus\{w\})}:w\in B_H~\mathrm{and}~ A\geq w \mathrm{~for~all~}A\in\mathcal{G}^0\setminus H\}.$$
	Then $X_1\cup X_2$ is the set of all graded prime  ideals of $L_K(\mathcal G)$.
\end{thm}
\begin{proof}
	Let $X$ be the set of all graded prime  ideals of $L_K(\mathcal G)$. We prove that $X=X_1\cup X_2$.
	
    Assume $I_{(H,B_H)}\in X_1$. We show that the zero ideal of $L_K\big(\mathcal{G}/(H,B_H)\big)$ is prime. Since $\{0\}$ is a graded ideal, it suffices to prove that for every nonzero graded  ideals $I,J$ of $L_K\big(\mathcal{G}/(H,B_H)\big)$, $IJ\ne\{0\}$ (see \cite[Proposition $\rm I\rm I.1.4$]{nas}). If $I$ and $J$ are such ideals, then by  Lemma \ref{lemgp}, there exist nonzero idempotents $q_{[A]}\in I$ and $q_{[B]}\in J$. Since $A,B\in\mathcal{G}^0\setminus H$ and $\mathcal{G}^0\setminus H$ is downward directed, there exists $C\in\mathcal G^0\setminus H$ such that $A,B\geq C$. Thus $q_{[C]}\in I\cap J$ and therefore $\{0\}$ is a prime ideal. Since $L_K\big(\mathcal{G}/(H,B_H)\big) \cong L_K(\mathcal G)/I_{(H,B_H)}$, we deduce that $I_{(H,B_H)}$ is prime. Hence $X_1\subseteq X$.
    
    Let $I_{(H,B_H\setminus\{w\})}\in X_2$. Then $\mathcal{G}^0\setminus H$ is downward directed, because $A\geq w$ for all $A\in\mathcal{G}^0\setminus H$. Now, in a similar way as before we obtain that $I_{(H,B_H\setminus\{w\})}$ is prime. Thus $X_2\subseteq X$.
    
    To establish the reverse inclusion, let $I_{(H,S)}\in X$. From Proposition \ref{probn} we have that either $S=B_H$ or $S=B_H\setminus\{w\}$ for some $w\in B_H$. If $S=B_H$, then, by Lemma \ref{lempd}, $\mathcal{G}^0\setminus H$ is downward directed. Hence $I_{(H,S)}\in X_1$. 
    
    Let $S=B_H\setminus\{w\}$ and $A\in\mathcal{G}^0\setminus H$. We show that $A\geq w$. Clearly the result holds for $w\in A$, so let $w\notin A$. By the primeness of $L_K\big(\mathcal{G}/(H,S)\big)$, we must have $IJ\ne\{0\}$ for every nonzero ideals $I$ and $J$ of $L_K\big(\mathcal{G}/(H,S)\big)$. Thus
    $$L_K\big(\mathcal{G}/(H,S)\big)q_{[\overline A]}L_K\big(\mathcal{G}/(H,S)\big)q_{[w']}L_K\big(\mathcal{G}/(H,S)\big)\ne\{0\},$$
    and hence $q_{[\overline A]}L_K\big(\mathcal{G}/(H,S)\big)q_{[w']}\ne \{0\}$. Since $[w']$ is a sink and $[\overline A]\cap[w']=[\emptyset]$, there exists a path $\alpha$ of positive length such that $q_{[\overline A]}t_{\alpha}q_{[w']}\ne0$. Therefore $s(\alpha)\subseteq[A]$ and $[w']\subseteq r(\alpha)$, which implies that $s_{\mathcal G}(\alpha)\in A$ and $w\in r_{\mathcal G}(\alpha)$. Thus $A\geq w$ and consequently $I_{(H,S)}\in X_2$. This proves that $X\subseteq X_1\cup X_2$.
\end{proof}

From \cite[Theorem 5.4]{ima}, we know that the ultragraph $\mathcal G$ satisfies Condition (K) if and only if every ideal of $L_K(\mathcal G)$ is graded. The following corollary now follows from \cite[Theorem 5.4]{ima} and Theorem \ref{thmgp}.

\begin{cor}
	If the ultragraph $\mathcal{G}$ satisfies Condition (K), then Theorem \ref{thmgp} gives a complete description of the prime ideals of $L_K(\mathcal G)$.
\end{cor}

\subsection{Non-graded prime ideals}
Let $\mathcal{G}/(H,S)$ be a quotient ultragraph and let $\gamma=e_1e_2\cdots e_n$ be a path of positive length. If $s(\gamma)\subseteq r(\gamma)$, then $\gamma$ is called a \emph{loop}. Denote $\gamma^1:=\{e_1,e_2,\ldots,e_n\}$. We say that $\gamma$ has an \emph{exit} if either $r(e_i)\neq s(e_{i+1})$ for some $1\leq i\leq n$ or there exist an edge $f\in\Phi(\mathcal{G}^1)$ and an index $i$ such that $s(f)\subseteq r(e_i)$ but $f\ne e_{i+1}$. The quotient ultragraph $\mathcal{G}/(H,S)$ satisfies \emph{Condition (L)} if every loop  in $\mathcal{G}/(H,S)$ has  an exit.

\begin{lem}\label{lemme}
	Let $\mathcal{G}/(H,S)$ be a quotient ultragraph. If $\gamma=e_1e_2\cdots e_n$ is a loop in $\mathcal{G}/(H,S)$ without exits, then $I_{\gamma^0}$ and  $K[x,x^{-1}]$ are Morita equivalent as rings, where $I_{\gamma^0}$ is an ideal of $L_K\big(\mathcal{G}/(H,S)\big)$ generated by $\{q_{s(e_i)}:1\leq i\leq n\}$.
\end{lem}
\begin{proof}
	Let $\gamma=e_1e_2\cdots e_n$ be a loop with no exists in $\mathcal{G}/(H,S)$ and $[v]=s(\gamma)$. Consider 
	$$\big(I_{\gamma^0},q_{[v]}I_{\gamma^0}q_{[v]}, I_{\gamma^0}q_{[v]}, q_{[v]}I_{\gamma^0},\psi,\phi\big),$$
	where $\psi(m\otimes n)=mn$ and $\phi(n\otimes m)=nm$. It can be shown that this is a (surjective) Morita context. Thus $I_{\gamma^0}$ and  $q_{[v]}I_{\gamma^0}q_{[v]}$ are Morita equivalent.
	
	Now we show that $q_{[v]}I_{\gamma^0}q_{[v]}\cong K[x,x^{-1}]$. Since $\gamma$ is a loop without exits, we deduce that
	$$I_{\gamma^0}=\mathrm{span}_K\big\{t_\mu q_{s(e_i)}t_{\nu^*}:s(e_i)\subseteq r(\mu)\cap r(\nu) ~\mathrm{and}~ 1\leq i\leq n \big\},$$
	and $q_{s(e_i)}=(t_{e_i}\cdots t_{e_n})q_{[v]}(t_{e_n^*}\cdots t_{e_i^*})$ for $1\leq i\leq n$. This implies that 
	$$q_{[v]}I_{\gamma^0}q_{[v]}=\mathrm{span}_K\big\{(t_\gamma)^m (t_{\gamma^*})^n:m,n\geq1\big\}.$$
	
	Let $E$ be the graph with one vertex $w$ and one loop $f$. Then  $\{p_w:=q_{[v]},s_f:=t_\gamma,s_{f^*}:=t_{\gamma^*}\}$ is a Leavitt $E$-family in $q_{[v]}I_{\gamma^0}q_{[v]}$. It follows from the universality of $L_K(E)$ and the graded uniqueness theorem that $L_K(E)\cong q_{[v]}I_{\gamma^0}q_{[v]}$. Since $L_K(E)\cong K[x,x^{-1}]$, we conclude that $I_{\gamma^0}$ and  $K[x,x^{-1}]$ are Morita equivalent.
\end{proof}

To prove the next lemma, we use the fact that $L_K\big(\mathcal{G}/(H,S)\big)$ can be estimated by the Leavitt path algebras of finite graphs. So, we recall this approximation from \cite[Section 3]{ima} and then we prove Lemma \ref{lemci}.
 
Let $\mathcal{G}/(H,S)$ be a quotient ultragraph and $F$ be a finite subset of $\Phi_{\mathrm{sg}}(G^0) \cup \Phi(\mathcal{G}^1)$. We construct a finite graph $G_F$ as follows. Let $F^0:=F\cap \Phi_{\mathrm{sg}}(G^0)$ and $F^1:=F\cap \Phi(\mathcal{G}^1)=\{e_1,\ldots,e_n\}$. For every $\omega=(\omega_1,\ldots, \omega_n)\in \{0,1\}^n\setminus \{0^n\}$, we define  $r(\omega):=\bigcap_{\omega_i=1}r(e_i)\setminus \bigcup_{\omega_j=0}r(e_j)$ and $R(\omega):=r(\omega)\setminus \bigcup_{[v]\in F^0}[v]$ which belong to $\Phi(\mathcal G^0)$. Set
\begin{multline*}
\Gamma_0:=\{\omega\in \{0,1\}^n\setminus\{0^n\}: $ where  vertices  $ [v_1],\ldots,[v_m] $ exist$\\
$ such  that $ R(\omega)=\bigcup_{i=1}^m[v_i] $ and $\emptyset\neq s^{-1}([v_i])\subseteq F^1 $ for $ 1\leq i\leq m\},
\end{multline*}
and
\[\Gamma_F:=\left\{\omega\in \{0,1\}^n\setminus \{0^n\}: R(\omega)\neq [\emptyset] \mathrm{~and ~} \omega\notin \Gamma_0 \right\}.\]

Define the finite graph $G_F=(G_F^0,G_F^1,r_F,s_F)$, where
\begin{align*}
G_F^0:=&F^0 \cup F^1 \cup \Gamma_F,\\
G_F^1:=&\left\{(e,f)\in F^1\times F^1: s(f)\subseteq r(e) \right\}\\
&\cup \left\{(e,[v])\in F^1\times F^0: [v]\subseteq r(e) \right\}\\
&\cup \left\{(e,\omega)\in F^1\times \Gamma_F: \omega_i=1 \mathrm{~~whenever~~} e=e_i \right\},
\end{align*}
with $s_F(e,f)=s_F(e,[v])=s_F(e,\omega)=e$ and $r_F(e,f)=f$, $r_F(e,[v])=[v]$, $r_F(e,\omega)=\omega$. By \cite[Lemma 3.1]{ima}, the elements
$$\begin{array}{lll}
P_e:=t_et_{e^*}, &P_{[v]}:=q_{[v]}(1-\sum\limits_{e\in F^1}t_et_{e^*}), &P_\omega:=q_{R(\omega)}(1-\sum\limits_{e\in F^1}t_et_{e^*}), \\
S_{(e,f)}:=t_eP_f, &S_{(e,[v])}:=t_e P_{[v]}, &S_{(e,\omega)}:=t_e P_\omega,\\
S_{(e,f)^*}:=P_ft_{e^*}, &S_{(e,[v])^*}:=P_{[v]}t_{e^*}, &S_{(e,\omega)^*}:=P_\omega t_{e^*},
\end{array}$$
form a Leavitt $G_F$-family in $L_K\big(\mathcal{G}/(H,S)\big)$ such that
$$L_K(G_F)\cong L_K(S,P)=\langle q_{[v]},t_e,t_{e^*}:[v]\in F^0,e\in F^1\rangle.$$

\begin{rem}\label{remlc}
	Let $\gamma:=e_1e_2\cdots e_n$ be a loop in $\mathcal{G}/(H,S)$ and $F$ be a finite subset of $\Phi_{\mathrm{sg}}(G^0) \cup \Phi(\mathcal{G}^1)$ containing $\{e_1,e_2,\ldots,e_n\}$. Then $\widetilde\gamma:=(e_1,e_2) \cdots(e_n,e_1)$ is a loop in $G_F$. Since the elements of $F^0\cup\Gamma_F$ are sinks in $G_F$, every loop in $G_F$ is of the form $\widetilde\beta$ where $\beta$ is a loop in $\mathcal{G}/(H,S)$. By using the argument of \cite[Lemma 4.8]{lar2}, we can show that $\gamma$ is a loop without exits in $\mathcal{G}/(H,S)$ if and only if $\widetilde\gamma$ is a loop without exits in $G_F$.
\end{rem} 

The set of vertices in the loops without exits of $\mathcal{G}/(H,S)$ is denoted by $P_c(\mathcal{G}/(H,S))$. Also, we denote by $I_{P_c(\mathcal{G}/(H,S))}$ the ideal of $L_K\big(\mathcal{G}/(H,S)\big)$ generated by the idempotents associated to the vertices in  $P_c(\mathcal{G}/(H,S))$.

\begin{lem}\label{lemci}
	Let $\mathcal{G}/(H,S)$ be a quotient ultragraph. If $z\in L_K\big(\mathcal{G}/(H,S)\big)\setminus I_{P_c(\mathcal{G}/(H,S))}$, then there exist $x,y\in L_K\big(\mathcal{G}/(H,S)\big)$ and $[\emptyset]\ne[A]\in \Phi(\mathcal G^0)$ such that $xzy=q_{[A]}$.
\end{lem}
\begin{proof}
	Let $\{F_n\}$ be an increasing sequence of finite subsets of $(G^H_S)^0_{\mathrm{sg}} \cup \Phi(\mathcal{G}^1)$ such that $\cup_{n=1}^\infty F_n=\Phi_{\mathrm{sg}}(G^0) \cup \Phi(\mathcal{G}^1)$. Then 
	$$\bigcup\nolimits_n L_K(G_{F_n}) =\langle  q_{[v]},t_e,t_{e^*}:[v]\in\Phi_{\mathrm{sg}}(G^0),~e\in\Phi(\mathcal{G}^1)\rangle= L_K\big(\mathcal{G}/(H,S)\big).$$
	
	Now, let $z\in L_K\big(\mathcal{G}/(H,S)\big)\setminus I_{P_c(\mathcal{G}/(H,S))}$. There exists $F_n$ such that $z\in L_K(G_{F_n})$. We show that $z\notin I_{P_c(G_{F_n})}$. By Remark \ref{remlc}, we have
	$$P_c(G_{F_n})=\{\widetilde\alpha:\alpha\in P_c(\mathcal{G}/(H,S))~ \mathrm{and}~ \gamma^1 \subseteq G_{F_n}^1\}.$$
    Suppose that $\gamma:=e_1e_2\cdots e_n\in P_c(\mathcal{G}/(H,S))$. For every $1\leq i\leq n$ we have $P_{e_i}=t_{e_i}t_{e_i^*}\in I_{P_c(\mathcal{G}/(H,S))}$. Thus $I_{P_c(\mathcal{G}/(H,S))}$ contains generators of $I_{P_c(G_{F_n})}$. This implies that $I_{P_c(G_{F_n})}\subseteq L_K(G_{F_n})\cap I_{P_c(\mathcal{G}/(H,S))}$ and consequently $z\in L_K(G_{F_n})\setminus I_{P_c(G_{F_n})}$.
	
	By \cite[Proposition 5.2]{arp}, there exist $x',y'\in L_K(G_{F_n})$ and $w\in G_{F_n}^0$ such that $x'zy'=P_w$. We distinguish three cases.
	\begin{enumerate}
		\item Let $w=[v]\in F_n^0$. Since $[v]\in\Phi_{\mathrm{sg}}(G^0)$, there exists $f\in\Phi(\mathcal{G}^1)\setminus F_n^1$ such that $[v]=s(f)$. Set $x= t_{f^*}x'$ and $y=y't_f$. Then $xyz=q_{r(f)}$.
		\item If $w=e\in F_n^1$, then $t_{e^*}x'zy't_e=q_{r(e)}$.
		\item Let $w=\omega\in\Gamma_{F_n}$. Thus there exists a vertex $[v]\subseteq R(\omega)$ such that either $[v]$ is a sink or there is an edge $f\in\Phi(\mathcal{G}^1)\setminus F_n^1$ with $s(f)=[v]$. In the former case, we deduce that $q_{[v]}x'zy'=q_{[v]}Q_{\omega}=q_{[v]}$ and in the later case $t_{f}^*x'zy't_f=t_{f}^*Q_\omega t_f=q_{r(f)}$.
	\end{enumerate}
\end{proof}
 
Proposition \ref{proli} is an immediate consequence of Lemma \ref{lemci}.

\begin{prop}\label{proli}
	Let $\mathcal{G}/(H,S)$ be a quotient ultragraph. If $I$ is an ideal of $L_K\big(\mathcal{G}/(H,S)\big)$  with $\{[A]\neq[\emptyset]:q_{[A]}\in I\}=\emptyset$, then $I\subseteq I_{P_c(\mathcal{G}/(H,S))}$.
\end{prop}

\begin{rem}
	Let $\mathcal{G}/(H,S)$ be a quotient ultragraph, $x\in L_K\big(\mathcal{G}/(H,S)$ and let the ideal of $L_K\big(\mathcal{G}/(H,S)$ generated by $I_{(H,S)}\cup\{x\}$ is denoted by $I_{\langle H,S,x \rangle}$. Suppose that $\gamma$ is a loop in $\mathcal{G}/(H,S)$ without exits and $f(x)$ is a polynomial in $K[x,x^{-1}]$. It can be shown that $I_{\langle f(t_\gamma) \rangle}= I_{\langle f(t_{\gamma'}) \rangle}$, where  $\gamma'$ is a permutation of $\gamma$.
\end{rem}
  
Let $H$ be a saturated hereditary subset of $\mathcal{G}^0$ and $\gamma=e_1e_2\cdots e_n$ be a loop in $\mathcal{G}$. We say that $\gamma=e_1e_2\cdots e_n$ is a loop in $\mathcal{G}^0\setminus H$ if $r_{\mathcal{G}}(\gamma)\in\mathcal{G}^0\setminus H$. Also,  $\gamma$ has \emph{an exit in} $\mathcal{G}^0\setminus H$ if either $r_{\mathcal{G}}(e_i)\setminus s_{\mathcal{G}}(e_{i+1})\in\mathcal{G}^0\setminus H$ for some $1\leq i\leq n$ or there exists an edge $f\in\mathcal{G}^1$ and an index $i$ such that $r_{\mathcal{G}}(f)\in \mathcal{G}^0\setminus H$ and $s_{\mathcal{G}}(f)\subseteq r_{\mathcal{G}}(e_i)$ but $f\ne e_{i+1}$. One can see that the quotient ultragraph  $\mathcal{G}/(H,B_H)$ satisfies Condition (L) if and only if every loop in $\mathcal{G}^0\setminus H$ has an exit in $\mathcal{G}^0\setminus H$. 
 
\begin{thm}\label{thmnp}
	Let $\mathcal{G}$ be an ultragraph and let $I$ be an ideal of $L_K(\mathcal G)$. Denote $H:=H_I$. Then $I$ is a non-graded prime ideal if and only if
	\begin{enumerate}
		\item$\mathcal{G}^0\setminus H$ is downward directed,
		\item $\mathcal{G}^0\setminus H$ contains a loop $\gamma$	without exits in $\mathcal{G}^0\setminus H$ and
		\item $I=I_{\langle H,B_H,f(s_\gamma)\rangle}$, where $f(x)$ is an irreducible polynomial in $K[x,x^{-1}]$.
	\end{enumerate}	    
\end{thm}
\begin{proof}
	Let $I$ be a non-graded prime ideal and $S:=\{w\in B_H:p_w^H\in I\}$. It follows from Lemma \ref{lempd} that $\mathcal{G}^0\setminus H$ is downward directed. Consider the quotient $L_K\big(\mathcal{G}/(H,S)\big) \cong L_K(\mathcal G)/I_{(H,S)}$ and let $\widetilde I$ be the image of $I$ under the quotient map. Since $I$ is a non-graded ideal, we have that $\widetilde I\ne\{0\}$. The argument in the proof of \cite[Theorem 4.4]{ima} implies that 
	$$\big\{[A]\neq[\emptyset]:[A]\in\Phi(\mathcal G^0)~\mathrm{and}~q_{[A]}\in \widetilde I\big\}=\emptyset.$$ 
	In view of Lemma \ref{lemgp}, this means that $\widetilde I$ is a non-graded ideal. Thus by the Cuntz-Krieger uniqueness theorem \cite[Theorem 3.6]{ima}, $\mathcal{G}/(H,S)$ contains a loop $\gamma$ without exits. Let $w\in B_H\setminus S$ and $X:=L_K\big(\mathcal{G}/(H,S)\big)/\widetilde I$. Since $\gamma$ has no exits and  $[w']$ is a sink, we get that $(q_{[w']}+\widetilde I)X(q_{s(\gamma)}+\widetilde I)=\{0\},$ in contradiction with the primeness of $X$. Hence $S=B_H$.
	
	Since $\mathcal{G}^0\setminus H$ is downward directed, we get that $\gamma$ is  unique (up to permutation). Thus, by Proposition \ref{proli}, we have $\widetilde I\subset I_{\gamma^0}$. From Lemma \ref{lemme}, we know that $I_{\gamma^0}$ is Morita equivalent to $K[x,x^{-1}]$ by the Morita correspondence $J\mapsto q_{s(\gamma)}Jq_{s(\gamma)}$. Since the primeness is preserved by the Morita correspondence, we deduce that $q_{s(\gamma)}\widetilde Iq_{s(\gamma)}$ is a prime ideal of $K[x,x^{-1}]$. Thus there exists an irreducible polynomial $f(x)$ in $K[x,x^{-1}]$ such that $q_{s(\gamma)}\widetilde Iq_{s(\gamma)}$ is generated by $f(x)$. Hence $\widetilde I=I_{\langle f(t_\gamma) \rangle}$. Since $\widetilde I\subset L_K\big(\mathcal{G}/(H,B_H)\big) \cong L_K(\mathcal G)/I_{(H,B_H)}$ and
	\begin{equation*}\begin{array}{ll}
	q_{[A]}=p_A+I_{(H,B_H)} & \mathrm{for}~~A\in\Phi(\mathcal G^0), \\
	t_{e}=s_e+I_{(H,B_H)} &  \mathrm{for}~~e\in\Phi(\mathcal{G}^1),\\
	t_{e^*}=s_{e^*}+I_{(H,B_H)} &  \mathrm{for}~~e\in\Phi(\mathcal{G}^1),
	\end{array}
	\end{equation*}
	we deduce that $I=I_{\langle H,B_H,f(s_\gamma)\rangle}$.
	
	Conversely, assume that the above three conditions hold. Denote by $\widetilde I$  the image of $I$ in the quotient $L_K(\mathcal G)/I_{(H,B_H)}$. Hence $\widetilde I=I_{\langle f(t_\gamma) \rangle}$. Since $\mathcal{G}^0\setminus H$ is downward directed, $\gamma$ is  unique (up to permutation). It follows from Proposition \ref{proli} that $\widetilde I\subset I_{\gamma^0}$. Thus $q_{s(\gamma)}\widetilde Iq_{s(\gamma)}$ is an ideal of $K[x,x^{-1}]$ generated by $f(x)$. As $f(x)$ is an irreducible
	polynomial in $K[x,x^{-1}]$, the ideal $q_{s(\gamma)}\widetilde Iq_{s(\gamma)}$ is prime and therefore, by Lemma \ref{lemme}, $\widetilde I$ is a prime ideal of $L_K(\mathcal G)/I_{(H,B_H)}$. Since $\{0\}\ne\widetilde I \subset I_{\gamma^0}$, we get that $\widetilde I$ is a non-graded ideal. Consequently, $I$ is a  non-graded prime ideal of $L_K(\mathcal G)$.
\end{proof}

%%%%%%%%%%%%%%%%%%%%%%%%%%%%%%%%%%%%%%%%%%%%%%%%%%%%%%%%%%%%%%%%%%%%%%
%%%%%%%%%%%%%%%%%%%%%%%%%%%%%%%%%%%%%%%%%%%%%%%%%%%%%%%%%%%%%%%%%%%%%%
\section{Primitive ideals}\label{s4}

In this section, we characterize primitive ideas of $L_K(\mathcal G)$. We determine primitive quotient ultragraph Leavitt path algebras and then we characterize graded primitive ideals. Finally, We see that the non-graded prime ideals of $L_K(\mathcal G)$ are always primitive.

\subsection{Graded primitive ideals}
A ring $R$ is said to be left (right) \emph{primitive}, if it has a faithful simple left (right) $R-$module. Since $L_K\big(\mathcal{G}/(H,S)\cong L_K\big(\mathcal{G}/(H,S)^{op}$, we deduce that $L_K\big(\mathcal{G}/(H,S)$  is left primitive if and only if it is right primitive. So we simply say it is primitive.

\begin{lem}\label{lemdq}
	Let $\mathcal{G}/(H,B_H)$ be a quotient ultragraph and
    $$X:=\big\{[A]\in\Phi(\mathcal G^0)\setminus\{[\emptyset]\}:[A]\cap[v]=[\emptyset]~\mathrm{for}~\mathrm{every}~\mathrm{vertex}~ [v]\in \Phi(G^0)\big\}.$$ 
    If $\mathcal{G}^0\setminus H$ is downward directed, then $X$ is closed under finite intersections. 
\end{lem}
\begin{proof}
	Suppose that $\mathcal{G}^0\setminus H$ is downward directed. We show that if $[A]\in X$ and $A\geq C$ for some $C\in\mathcal{G}^0\setminus H$, then  $C\subseteq A$. If $C\nsubseteq A$, then there is a path $\alpha$ of positive length such that $s_{\mathcal{G}}(\alpha)\in A$ and $C\subseteq r_{\mathcal{G}}(\alpha)$. So $\{s_{\mathcal{G}}(\alpha)\} \in H$. Thus, by the hereditary property of $H$, $C\in H$, which is impossible.
	
	Now, let $[A_1],\ldots,[A_n]\in X$ and $A:=A_1\cap\cdots\cap A_n$. Set $C_1=A_1$. Since $\mathcal{G}^0\setminus H$ is downward directed, there exists $C_2\in \mathcal{G}^0\setminus H$ such that $C_1,A_2\geq C_2$. Therefore $C_2\subseteq C_1\cap A_2$. Similarly, there exists $C_3\in \mathcal{G}^0\setminus H$ such that $C_3\subseteq C_2\cap A_2$. Repeating this process we find $C_n\in\mathcal{G}^0\setminus H$ such that $C_n\subseteq A$. Hence $A\in \mathcal{G}^0\setminus H$ and thus $[A]\ne[\emptyset]$. Since $[A]\cap[v]=[\emptyset]$ for all $[v]\in \Phi(G^0)$, we conclude that $[A]\in X$.
\end{proof}

Let $A$ be a unital ring. By \cite[Theorem 1]{for}, $A$ is left primitive if and only if there is a left ideal $M\ne A$ of $A$ such that for every nonzero two-sided ideal $I$ of $A$, $M+I=A$. We use this to prove the following theorem.

\begin{thm}\label{thmgpi}
	Let $\mathcal{G}/(H,S)$ be a quotient ultragraph. Then $L_K\big(\mathcal{G}/(H,S)\big)$ is primitive if and only if  one of the following holds:
	\begin{enumerate}
		\item[(i)] $S=B_H$, $\mathcal{G}/(H,S)$ satisfies Condition (L) and $\mathcal{G}^0\setminus H$ is downward directed.
		\item[(ii)] $S=B_H\setminus\{w\}$ for some $w\in B_H$ and $A\geq w$ for all $A\in\mathcal{G}^0\setminus H$.
	\end{enumerate}	
\end{thm}
\begin{proof}
	First, suppose that $L_K\big(\mathcal{G}/(H,S)\big)$ is primitive. Then $L_K\big(\mathcal{G}/(H,S)\big)$ is prime and thus $I_{(H,S)}$ is a prime ideal of $L_K(\mathcal G)$. From Proposition \ref{probn} we deduce that $|B_H\setminus S|\leq1$. So either  $|B_H\setminus S|=1$ or $S=B_H$. Thus (ii) follows by Theorem \ref{thmgp} or (i) holds by Theorem \ref{thmgp} and appalling Lemma \ref{lemme}, respectively. 
	
	Conversely, suppose that (i) holds. Applying Theorem \ref{thmgp}, the graded ideal $I_{(H,S)}$ is prime which implies that   $L_K\big(\mathcal{G}/(H,S)\big)$ is a prime ring.  By \cite[Lemmas 2.1 and 2.2]{lan}, there exists a prime unital
	$K$-algebra $R$ such that $L_K\big(\mathcal{G}/(H,S)\big)$ embeds in $R$ as a two-sided ideal and primitivity of  $R$ gives  the primitivity of $L_K\big(\mathcal{G}/(H,S)\big)$ and vice versa. So it is enough to show that $R$ is primitive. Suppose that $X$ is the set defined in Lemma \ref{lemdq}. We distinguish two cases depending on $X$. 
	
    Case 1. Let $X=\emptyset$ and let $[v]$ be a vertex in $\mathcal{G}/(H,S)$. Define $H=\big\{[u]\in\Phi(G^0):v\geq u\big\}$. Since $\Phi(G^0)$ is  countable we can write $H=\big\{[v_1],[v_2],\ldots\big\}$. We claim that there exists a sequence $\{\lambda_i\}_{i=1}^{\infty}$ in $\mathrm{Path}\big(\mathcal{G}/(H,S)\big)$ such that for every $i\in\mathbb N$, $\lambda_{i+1}=\lambda_i\mu_i$ for some path $\mu_i\in \mathrm{Path}\big(\mathcal{G}/(H,S)\big)$ and also $v_i\geq s_{\mathcal G}(\mu_i)$. Note that for every $[A]\in\Phi(\mathcal G^0)$ we define $s_{\mathcal G}([A])=r_{\mathcal G}([A])=A$.  
      
    Set $\lambda_1=[v_1]$, since $\mathcal{G}/(H,S)$ is downward directed, there exists $A_{2}\in\mathcal G^0\setminus H$ such that $v_{2}\geq A_{2}$ and $r_{\mathcal G}(\lambda_{1})\geq A_{2}$. If $A_{2}\subseteq r_{\mathcal G}(\lambda_{1})$, then  $A_2=\{v_1\}$ and take $\mu_{1}=[A_2]$. If $A_{2}\nsubseteq r_{\mathcal G}(\lambda_{1})$, then there is a path $\mu_{1}$ of positive length such that $s_{\mathcal G}(\mu_{1})\in r_{\mathcal G}(\lambda_{1})$ and $A_2\subseteq r_{\mathcal G}(\mu_{1})$. Now define $\lambda_2=\lambda_{1}\mu_1$. Also $v_1\geq s_{\mathcal G}(\mu_{1})$, $v_{2}\geq A_{2}$ and $A_{2}\subseteq r_{\mathcal G}(\lambda_{2})$. Now by induction we assume that there exist $\lambda_1,\ldots,\lambda_{n}$ with the previous properties. Corresponding there exist $A_{k}\in\mathcal G^0\setminus H$ for $k\in\{1,2,\ldots,n\}$ such that $v_{k}\geq A_{k}$ and $A_{k}\subseteq r_{\mathcal G}(\lambda_{k})$. Since $[A_n]\ne[\emptyset]$ and $X=\emptyset$, there is a vertex $[u_n]\in\Phi(G^0)$ such that $u_n\subseteq A_n$. Hence $v_{n}\geq u_{n}$ and $[u_n]\subseteq r(\lambda_{n})$, one can show that there is a path $\mu_{n}\in\mathrm{Path}(\mathcal{G})$ such that $s_{\mathcal G}(\mu_{n})=u_n$. Now define $\lambda_{n+1}=\lambda_{n}\mu_n$ and corresponding $A_{n+1}\in\mathcal G^0\setminus H$ with the same property as before.
     
    Now set
    $$M=\sum\limits_{i=1}^{\infty} R(1-t_{\lambda_i}t_{\lambda_i^*}),$$
    which is a left ideal of $R$. If $1\in M$, then  $1=\sum\nolimits_{i=1}^{n} r_i(1-t_{\lambda_i}t_{\lambda_i^*})$ for some $n\in\mathbb N$ and $r_1,\ldots,r_n\in R$.  Since $t_{\lambda_j}t_{\lambda_j^*}t_{\lambda_k}t_{\lambda_k^*}=t_{\lambda_k}t_{\lambda_k^*}\quad(k\geq j)$, we have  $$t_{\lambda_n}t_{\lambda_n^*}=\sum\nolimits_{i=1}^{n} r_i(1-t_{\lambda_i}t_{\lambda_i^*})t_{\lambda_n}t_{\lambda_n^*}=0,$$
    which is a contradiction. Hence $1\notin M$ and $M$ is a proper
    left ideal of $R$.
    
    By \cite[Theorem 1]{for}, to prove that $R$ is primitive,  It is enough to show that $M+I=R$ for every nonzero two-sided ideal $I$ of $R$. Take $I_1=I\cap L_K\big(\mathcal{G}/(H,S)\big)$. But $R$ is prime,  so $I_1$ is a nonzero two sided ideal of $L_K\big(\mathcal{G}/(H,S)\big)$. Since $\mathcal{G}/(H,S)$ satisfies Condition (L), by the Cuntz-Krieger uniqueness theorem \cite[Theorem 3.6]{ima}, there exists $[\emptyset]\ne[A]\in\Phi(\mathcal G^0)$ such that $q_{[A]}\in I_1$. By downward directedness, there exists $C\in\mathcal G^0\setminus H$ for which $v\geq C$ and $A\geq C$. As $X=\emptyset$, there is a vertex $[z]\in\Phi( G^0)$ such that  $z\in C$. Hence $v\geq z$ and thus $z=v_n$ for some $n\geq1$. We know that $v_n\geq s_{\mathcal G}(\mu_n)$, where $\lambda_{n+1}=\lambda_n\mu_n$. Consequently, $A\geq C\geq v_n\geq s_{\mathcal G}(\mu_n)\geq r_{\mathcal G}(\lambda_{n+1})$. Since $q_{[A]}\in I_1$, we deduce that $q_{r(\lambda_{n+1})}\in I_1$ and $t_{\lambda_{n+1}}t_{\lambda^*_{n+1}}\in I_1$. Thus
    $$1=(1-t_{\lambda_{n+1}}t_{\lambda^*_{n+1}})+t_{\lambda_{n+1}}t_{\lambda^*_{n+1}}\in M+I,$$
    which implies that $M+I=R$.
    
    Case 2. Let $X\ne\emptyset$. Define
    $M=\sum\limits_{[A]\in X} R(1-q_{[A]}),$ which is a left ideal of $R$. We claim that $M$ is proper, otherwise let $1=\sum\nolimits_{i=1}^{n} r_i(1-q_{[A_i]})\in M$, where $[A_1],\ldots,[A_n]\in X$ and $r_1,\ldots,r_n\in R$. Then $q_{[A]}=\sum\nolimits_{i=1}^{n} r_i(1-q_{[A_i]})q_{[A]}=0$, where $A:=A_1\cap\cdots\cap A_n$. By Lemma \ref{lemdq}  $[A]\neq[\emptyset]$, so $q_{[A]}\ne0$, a contradiction.
       
    Consider  $I$ and $I_1$ as before. Choose $[A]\in X$, $C\in\mathcal G^0\setminus H$ and $[\emptyset]\ne[B]\in\Phi(\mathcal G^0)$ such that $q_{[B]}\in I_1$, $B\geq C$ and $A\geq C$. As we have seen in the proof of Lemma \ref{lemdq}, $[C]\in X$. Since $B\geq C$, we have $q_{[C]}\in I_1$. Therefore $1=(1-q_{[C]})+q_{[C]}\in M+I$. Hence $M+I=R$  and consequently $R$ is primitive. 
    
    Finally, suppose that (ii) holds. Then $\mathcal{G}^0\setminus H$ is downward directed. By Theorem \ref{thmgp}, $L_K\big(\mathcal{G}/(H,S)\big)$ is a prime ring. We claim that $\mathcal{G}/(H,S)$ satisfies Condition (L), now with a similar argument as in the previous part if we take $M=R(1-q_{[w']})$), one can show that $M+I=R$ for every nonzero two-sided ideal $I$ of $R$ and consequently $R$ is primitive. 
      
    To prove the claim,  let $\gamma=e_1e_2\cdots e_n$ be a loop in $\mathcal{G}/(H,S)$.  By (ii) we have $s_{\mathcal G}(\gamma)\geq w$, so either $w=s_{\mathcal G}(\gamma)$ or $w\ne s_{\mathcal G}(\gamma)$. Thus  $r(e_n)\ne s(e_1)$ or  there is a path $\alpha$ such that $s_{\mathcal G}(\alpha)=s_{\mathcal G}(\gamma)$ and $\{w\}\subseteq r_{\mathcal G}(\alpha)$, respectively. Therefore $\gamma$ has an exit in $\mathcal{G}/(H,S)$.
 \end{proof}

A two-sided ideal $I$ of a ring $R$ is called a left (right) \emph{primitive ideal} if $R/I$ is a left (right) primitive ring. We note that a graded ideal $I_{(H,S)}$ of $L_K(\mathcal G)$ is left primitive if and only if it is  right primitive.

Theorem \ref{thmgpi} immediately yields the following.

\begin{cor}\label{corpri}
	Let $\mathcal{G}$ be an  ultragraph. A graded ideal $I_{(H,S)}$ of $L_K(\mathcal G)$ is primitive if and only if  one of the following holds:
	\begin{enumerate}
		\item[(i)] $S=B_H$, $\mathcal{G}^0\setminus H$ is downward directed and every loop in $\mathcal{G}^0\setminus H$ has an exit in $\mathcal{G}^0\setminus H$.
		\item[(ii)] $S=B_H\setminus\{w\}$ for some $w\in B_H$  and $A\geq w$ for all $A\in\mathcal{G}^0\setminus H$.
	\end{enumerate}
\end{cor}

\subsection{Non-graded primitive ideals}
It is well known that every primitive ideal of a ring is prime. The next theorem shows that every non-graded prime ideal of $L_K(\mathcal G)$ is primitive.

\begin{thm}\label{thmnpr}
	Let $\mathcal{G}$ be an  ultragraph. A non-graded prime ideal of $L_K(\mathcal G)$ is prime if and only if it is primitive.
\end{thm}
	
\begin{proof}
	Let $I$ be a non-graded prime ideal of $L_K(\mathcal G)$. By Theorem \ref{thmnp}, $I=I_{\langle H,B_H,f(s_\gamma)\rangle}$, where $\mathcal{G}^0\setminus H$ is a downward directed set containing a loop $\gamma$ without exits in $\mathcal{G}^0\setminus H$ and $f(x)$ is an irreducible polynomial in $K[x,x^{-1}]$. Let $\widetilde I=I/I_{(H,B_H)}$ and $s_{\mathcal G}(\gamma)=v$. As in the proof of Theorem \ref{thmnp}, the ideal $q_{[v]}\widetilde Iq_{[v]}$ of $K[x,x^{-1}]$ generated by $f(x)$ is a maximal ideal. Hence $K[x,x^{-1}]/q_{[v]}\widetilde Iq_{[v]}\cong K$. On the other hand, we have
	$$\overline P_v\bigg(\frac{L_K(\mathcal G)}{I}\bigg)\overline P_v \cong \frac{q_{[v]}\bigg(\frac{L_K(\mathcal G)}{I_{(H,B_H)}}\bigg)q_{[v]}}{q_{[v]}\bigg(\frac{I}{I_{(H,B_H)}}\bigg)q_{[v]}}\cong \frac{q_{[v]}L_K\big(\frac{\mathcal{G}}{(H,B_H)}\big)q_{[v]}}{q_{[v]}\widetilde Iq_{[v]}} \cong \frac{K[x,x^{-1}]}{q_{[v]}\widetilde Iq_{[v]}},$$
	where $\overline P_v=p_v+I$ and $q_{[v]}=p_v+I_{(H,B_H)}$. Since every field is primitive, we see that $L_K(\mathcal G)/I$ has a primitive corner and so, $I$ is a primitive ideal of $L_K(\mathcal G)$.
\end{proof}

\begin{ex}
	Consider the ultragraph $\mathcal{G}$ given below.
	
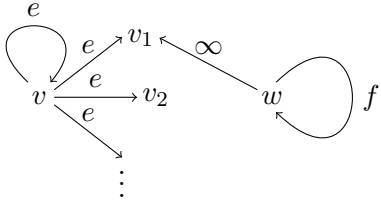
\begin{center}
		\begin{tikzpicture}
		\draw (1.2,0) node(x)  {$v$} (1.2,0);
		\draw[->] (1.05,.2) .. controls (0,1.2) and (2.2,1.2) ..node[above] {$e$}(1.35,.2);
		\draw [->] (1.4,.1) -- node[ above] {$e$} (2.3,.8);
		\draw (2.55,.79) node(x)  {$v_1$} (2.55,.79);
		\draw [->] (1.4,0) -- node[ above] {$e$} (2.5,0);
		\draw (2.75,0) node(x)  {$v_2$} (2.75,0);
		\draw [->] (1.4,-.1) -- node[ above] {$e$} (2.3,-.8);
		\draw (2.3,-1) node(x)  {$.$} (2.3,-1);
		\draw (2.3,-1.15) node(x)  {$.$} (2.3,-1.15);
		\draw (2.3,-1.3) node(x)  {$.$} (2.3,-1.3);
		\draw (4.3,0) node(x)  {$w$} (4.3,0);
		\draw[->] (4.35,.2) .. controls (5.7,1.5) and (5.7,-1.5) ..node[right] {$f$}(4.35,-.2);
		\draw [->] (4.1,.1) -- node[ above] {$\infty$} (2.8,.79);
		\end{tikzpicture}
	\end{center}
    Set $H_1=\overline{\{v\}}$ and $H_2=\overline{\{r_{\mathcal G}(e)\setminus\{v,v_1\}\}}$ (for $X\subseteq\mathcal G^0$, $\overline X$ denotes the smallest saturated hereditary subset of $\mathcal G^0$ containing $X$). We have that $B_{H_1}=\{w\}$ and $B_{H_2}=\emptyset$. Let $H$ be  a saturated hereditary subcollection of $\mathcal G$. It can be shown that $\mathcal G^0\setminus H$ is downward directed if and only if $H=H_1$ or $H=H_2$. It follows from Theorem \ref{thmgp} that $I_{(H_1,\{w\})}$ and $I_{(H_2,\emptyset)}$ are (graded) prime ideals of $L_K(\mathcal G)$. Since $A\geq w$ for every $A\in\mathcal{G}^0\setminus H_1$, by Corollary \ref{corpri} (ii), $I_{(H_1,\emptyset)}$ is a primitive ideal of $L_K(\mathcal G)$. $I_{(H_1,\{w\})}$ is not primitive because $f$ is a loop without exits
    in  $\mathcal{G}^0\setminus H_1$. If $g(x)$ is an irreducible polynomial in $K[x,x^{-1}]$, then, by Theorem \ref{thmnp} and Theorem \ref{thmnpr}, $I_{\langle H_1,\{w\},g(s_f)\rangle}$ is a non-graded prime (primitive) ideal in $L_K(\mathcal G)$. Since every loop in $\mathcal{G}^0\setminus H_2$ has an exit in $\mathcal{G}^0\setminus H_2$, by Corollary \ref{corpri} (i), $I_{(H_2,\emptyset)}$ is primitive. Finally, Let $X$ be the set of  prime ideals and $Y$ be the set of primitive ideals in $L_K(\mathcal G)$. Then we have
    $$X=\big\{ I_{(H_1,\{w\})},I_{(H_1,\emptyset)},I_{(H_2,\emptyset)}, I_{\langle H_1,\{w\},g(s_f)\rangle}\big\}$$
    and $Y=X\setminus\{I_{(H_1,\{w\})}\}$, where $g(x)$ is an irreducible polynomial in $K[x,x^{-1}]$.
\end{ex}


\begin{thebibliography}{99}
	

\bibitem{abr}{\it G. Abrams and G. Aranda Pino}, The Leavitt path algebra of a graph, J. Algebra, {\bf 293}(2005), 319-334.

\bibitem{abr1} {\it G. Abrams and G. Aranda Pino}, The Leavitt path algebras of arbitrary graphs, Houston J. Math., {\bf 34}(2008), 423-442.


\bibitem{abr2} {\it G. Abrams, J. P. Bell  and K. M. Rangaswamy},  On prime nonprimitive von Neumann regular algebras, Trans. Amer. Math. Soc.,
{\bf 366(5)}( 2014), 2375-2392.

\bibitem{arp} {\it G. Aranda Pino, J. Brox and M. Siles Molina}, Cycles in Leavitt path algebras by means of idempotents, Forum Math., {\bf 27}(2015), 601-633.

\bibitem{arp1} {\it G. Aranda Pino, E. Pardo  and M. Siles Molina}, Exchange Leavitt path algebras and stable rank, J. Algebra, {\bf 305}(2006),  912-936.

\bibitem{arp2} {\it G. Aranda Pino, E. Pardo  and M. Siles Molina}, Prime spectrum and primitive Leavitt path algebras, Indiana Univ. Math. J., {\bf 58(2)}(2009), 869-890.

\bibitem{cas} {\it G. G. de Castro, D.Gon\c{c}alves, and D. W. van Wyk}, Ultragraph algebras via labelled graph groupoids, with applications to generalized uniqueness theorems, J. Algebra, {\bf579}(2021), 456-495.

\bibitem{fir}{\it M. M. Firrisa}, Morita equivalence of graph and ultragraph Leavitt path algebras, arXiv:2006.06521, (2020).

\bibitem{for} {\it E. Formanek}, Group rings of free products are primitive, J. Algebra, {\bf 56}(1979), 395-398.

\bibitem{fow}{\it N. Fowler, M. Laca and I. Raeburn}, The $C^*$-algebras of infinite graphs, Proc. Amer. Math. Soc., {\bf 128(3)}(2000), 2319-2327.

\bibitem{gon} {\it D. Gon\c{c}alves and D. Royer}, Simplicity and chain conditions for ultragraph Leavitt path algebras via partial skew group ring theory, . Aust. Math. Soc., {\bf 109(3)}(2020), 299-319.

\bibitem{gon1} {\it D. Gon\c{c}alves and D. Royer}, Irreducible and permutative representations of ultragraph Leavitt path algebras, Forum Math., {\bf32}(2020), 41-431.

\bibitem{gon2} {\it D. Gon\c{c}alves and D. Royer}, Representations and the reduction theorem for ultragraph Leavitt path algebras, J. Algebraic Combin., {\bf53}(2021), 505-526.

\bibitem{haz}{\it R. Hazrat, T. G. Nam}, Realizing ultragraph Leavitt path algebras as Steinberg algebras, arXiv:2008.04668, (2020).

\bibitem{ima} {\it M. Imanfar, A. Pourabbas and H. Larki}, The leavitt path algebras of ultragraphs, Kyungpook Math. J., {\bf 60(1)}(2020), 21-43.


\bibitem{kul} {\it K. M. Rangaswamy}, The theory of prime ideals of Leavitt path algebras over arbitrary graphs, J. Algebra, {\bf 375}(2013), 73-96.


\bibitem{kum1} {\it A. Kumjian, D. Pask and I. Raeburn}, Cuntz-Krieger algebras of directed graphs,  Pacific J. Math., {\bf 184}(1998), 161-174.

\bibitem{lan} {\it C. Lanski, R. Resco and L. Small}, On the primitivity of prime rings, J. Algebra, {\bf 59(2)}(1979), 395-398.


\bibitem{lar1}{\it H. Larki}, Ideal structure of Leavitt path algebras with coefficients in a commutative ring with unit, Comm. Algebra, {\bf 43}(2015), 5031-5058.

\bibitem{lar2}{\it H. Larki}, {Primitive ideals and pure infiniteness of ultragraph $C^*$-algebras}, J. Korean Math. Soc., {\bf56}(2019) 1-23.

\bibitem{lea}{\it W.G. Leavitt}, Modules without invariant basis number, Proc. Amer. Math. Soc., {\bf 8}(1957), 322-328.

\bibitem{nam}{\it T. G. Nam and N. D. Nam}, Purely infinite simple ultragraph
Leavitt path algebras, arXiv:2007.08144, (2020).

\bibitem{nas}
{\it C. N$\check{\mathrm{a}}$st$\check{\mathrm{a}}$sescu and F. van Oystaeyen},  Graded Ring Theory, North Holland, Amsterdam, 1982.


\bibitem{tom1}
{\it M. Tomforde},  A unified approach to Exel-Laca algebras and $C^*$-algebras associated to graphs, J. Operator Theory, {\bf 50}(2003),  345-368.



\end{thebibliography}
\end{document}